\documentclass[11pt,a4paper]{article}

\usepackage{amsmath}
\usepackage{amsthm,amssymb}

\usepackage{graphicx} 

\newcommand{\qbinom}[3]{\genfrac{[}{]}{0pt}{}{#1}{#2}_{#3}}

\newtheorem{theorem}{Theorem}
\newtheorem{prop}[theorem]{Proposition}

\newtheorem{corollary}[theorem]{Corollary}

\theoremstyle{definition}

\begin{document}

\title{Enumeration of area-weighted Dyck paths with restricted height}
\author{A.L.~Owczarek$^1$ and T. Prellberg$^2$ \\[1ex]
\footnotesize
  \begin{minipage}{13cm} 
 $^1$Department of Mathematics and Statistics,\\
  The University of Melbourne,\\Parkville, Victoria 3010, Australia\\[1ex]
  $^2$ School of Mathematical Sciences\\
Queen Mary University of London\\
Mile End Road, London E1 4NS, UK
\end{minipage}
}

\maketitle

\begin{center}
  \rule{10cm}{1pt}
\end{center}

\begin{abstract}
We derive explicit expressions for $q$-orthogonal polynomials arising
in the enumeration of area-weighted Dyck paths with restricted height. 
\end{abstract}

\begin{center}
  \rule{10cm}{1pt}
\end{center}

\section{Introduction and Statement of Results}

Dyck paths are directed walks on $\mathbb{Z}^2$ starting at $(0,0)$ 
and ending on the line $y=0$, which have no vertices with negative
$y$-coordinates, and which have steps in the $(1,1)$ and $(1,-1)$
directions \cite{stanley1999}. We impose the additional geometrical constraint that the
paths have height at most $h$, {\it i.e.}, they lie between lines $y=0$ and
$y=h$. 
Given a Dyck path $\pi$, we define the length $n(\pi)$ to be half the number of
its steps, and the area $m(\pi)$ to be the sum of the starting heights of all steps
in the $(1,1)$ direction in the path. This is sometimes also called the rank function
of a Dyck path \cite{ferrari2005,sapounakis2006}, and is equivalent to the number of diamond
plaquettes under the Dyck path. An alternative definition of the area is the sum of the
heights of all steps in the Dyck path, which evaluates to $n(\phi)+2m(\phi)$ and is equivalent
to the number of triangular plaquettes under the Dyck path. The definition used here has the 
advantage of enabling a more elegant and concise mathematical formulation of our results.

We denote by $u(\pi)$ and $v(\pi)$ the number of vertices in the line $y=0$ 
(excluding the vertex $(0,0)$) and the number of vertices in the line $y=h$, respectively.
Let ${\mathcal{D}_h}$ be the set of Dyck paths with height at most $h$, and define
the generating function
$$
D_h(a,b;q,t)=\sum_{\pi\in\mathcal{D}_h} a^{u(\pi)}b^{v(\pi)}q^{m(\pi)}t^{n(\pi)}\;. 
$$
The generating function for Dyck paths with restricted height \cite{brak2005, flajolet1980}  and the generating function for area-weighed
Dyck paths \cite{duchon2000,flajolet1980} have previously been studied. Here we extend these works by combining these properties.

The purpose of this note is to derive the following expression for $D_h(a,b;q,t)$.
\begin{theorem} 
For $h\geq0$,
\begin{multline}
\label{result}
D_h(a,b;q,t)=\\
\frac
{{\displaystyle\sum\limits_{m=0}^\infty}
(-t)^m q^{m(m-1)}
\left( (1-b) \qbinom{h-m}{m}{q} 
    +b \qbinom{h+1-m}{m}{q} 
     - (1-b) \qbinom{h-1-m}{m-1}{q}
    - b \qbinom{h-m}{m-1}{q} \right)}
{{\displaystyle\sum\limits_{m=0}^\infty}
(-t)^m q^{m(m-1)}
\left( (1-b) \qbinom{h-m}{m}{q} 
    +b \qbinom{h+1-m}{m}{q} 
     - (1-a)(1-b) \qbinom{h-1-m}{m-1}{q}
    - (1-a) b \qbinom{h-m}{m-1}{q} \right)}\;.
\end{multline}
\end{theorem}

Here, we have used the standard notation for $q$-binomial coefficients \cite{goulden1983,stanley1997};
for $n\geq0$ and $0\leq k\leq n$ we define
$$
\qbinom{n}{k}{q}=\frac{(q;q)_n}{(q;q)_k (q;q)_{n-k}}\;,\quad\text{where}\quad (a;q)_n=\prod_{i=0}^{n-1}(1-aq^i)\;.
$$
We extend this definition to integer-valued $n,k$ by letting $\qbinom{n}{k}{q}=0$ when
$k<0$ or $k>n$. (This definition
is commensurate with the lattice path definition of $q$-binomial coefficients and is necessary to allow
for Theorem 1 to be valid even for $h=0$.)

For $a=b=1$, this identity simplifies considerably.
\begin{corollary}
For $h\geq0$,
$$
D_h(1,1;q,t)=
\frac
{{\displaystyle\sum\limits_{m=0}^\infty}
(-t)^m q^{m^2} \qbinom{h-m}{m}{q}}
{{\displaystyle\sum\limits_{m=0}^\infty}
(-t)^m q^{m(m-1)} \qbinom{h+1-m}{m}{q}}\;.
$$
\end{corollary}

Taking the limit $h\to\infty$, we recover the well-known result \cite{flajolet2009}
that the area-weighted generating function $D(q,t)$ for Dyck
paths without height restriction is given by
$$
D(q,t)=\frac
{{\displaystyle\sum\limits_{m=0}^\infty}
\dfrac{(-t)^m q^{m^2}}{(q;q)_m}}
{{\displaystyle\sum\limits_{m=0}^\infty}
\dfrac{(-t)^m q^{m(m-1)}}{(q;q)_m}}\;.
$$
More precisely, $D(q,t)$ as defined here is related to $F(z,q)$ in Eqn. (75) of Chapter V in \cite{flajolet2009} via $F(z,q)=D(q,qz)$.

\section{Proofs}

We use as the starting point of our derivation a well-established connection between lattice path 
enumeration and continued fractions \cite{flajolet1980}. 
\begin{prop} $D_0(a,b;q,t)=b$, $D_1(a,b;q,t)=1/(1-abt)$, and for $h\geq2$,
\begin{equation}
\label{confrac}
D_h(a,b;q,t)= \cfrac{1}{1- \cfrac{ at}{1-
    \cfrac{qt}{1- \cfrac{q^2t}{1- \cfrac{q^3t}{
          \genfrac{}{}{0pt}{}{}{\ddots 
            \genfrac{}{}{0pt}{}{}{
              \genfrac{}{}{0pt}{}{}{ 
              \genfrac{}{}{0pt}{}{}{ 
                \genfrac{}{}{0pt}{}{}{
                  \cfrac{q^{h-2}t}{1-bq^{h-1}t}}}}}}}}}}}\;.  
\end{equation}
\end{prop}

While this can easily be proved by specialising the general theory in \cite{flajolet1980} to
the case at hand, we shall give a direct combinatorial proof.

\begin{proof} 
The only Dyck path of height zero is the zero step Dyck path. If $h=0$ then it has weight $b$, whence
$D_0(a,b;q,t)=b$. Let now $h\geq1$.  Except for the zero-step Dyck path with weight $1$, every Dyck 
path of height at most $h$ 
can be decomposed uniquely into a Dyck path of height at most $(h-1)$ bracketed by a pair of 
steps into the $(1,1)$ and $(1,-1)$ directions, followed by another Dyck path of height $h$. 
The associated generating functions are $atD_{h-1}(1,b;q,qt)$ and $D_h(a,b;q,t)$, respectively. 
This decomposition leads to the functional-recurrence
$$D_h(a,b;q,t)=1+atD_{h-1}(1,b;q,qt)D_h(a,b;q,t)\;.$$
Iterating $D_h(a,b;q,t)=1/(1-atD_{h-1}(1,b;q,qt))$ gives (\ref{confrac}).
\end{proof}

It is clear that the generating function can also be written as a rational function, and
from Section 3 in \cite{flajolet1980} we obtain the following three-term recurrence.
\begin{prop} For $h\geq1$,
$$
D_h(a,b;q,t)= \frac{Q_h(0,b;q,t)}{Q_h(a,b;q,t)}\;,
$$
where
\begin{equation}
\label{recurrence}
Q_h(a,b;q,t)=\begin{cases}  1-abt\;, & h=1\;,\\
                             1-at-bqt\;, & h=2\;,\\
                             Q_{h-1}(a,1;q,t) - bq^{h-1}t Q_{h-2}(a,1;q,t)\; & h\geq 3\;.
\end{cases}
\end{equation}
\end{prop}

\begin{proof}
The initial conditions follow from $D_1(a,b;q,t)=1/(1-abt)$ and $D_2(a,b;q,t)=(1-bqt)/(1-at-bqt)$,
and the factor $bq^{h-1}t$ in the three-term recurrence is just the final term in the continued fraction
(\ref{confrac}). More precisely, comparing Eqn. (\ref{confrac}) with the $h$-th convergent of the $J$-fraction 
on page 152 of \cite{flajolet1980}, we have $t=z^2$, $a_0=a$, $a_k=1$ for $k\geq 1$, $b_k=q^{k-1}$ for $0\leq k<h$, $b_h=bq^{h-1}$,
and $c_k=0$ for $k\geq 0$. The linear recurrences given on page 152 of \cite{flajolet1980} then reduce to the recurrence in Eqn. (\ref{recurrence}).
\end{proof}

We proceed by considering the generating function of the denominators $Q_h(a,b;q,t)$,
$$
W(z;a,b;q,t) = \sum_{h=1}^{\infty} Q_h(a,b;q,t) z^h\;.
$$
The next proposition expresses $W(z;a,b;q,t)$ in terms of the basic hypergeometric series 
$\phi(z,q,t) ={}_1\phi_2(q;0,z;q,t)$ \cite{gasper2004}, {\em i.e.},
$$
\phi(z,q,t)=\sum_{n=0}^{\infty} \frac{q^{n(n-1)} t^n}{(z;q)_n}\;.
$$

\begin{prop}
\begin{multline}
\label{genfunct}
W(z;a,b;q,t)=
\frac{abt^2z^3-at^2z^3+abtz-abt-atz-btz-bz+b+z}{tz}\\
   +\frac{(bz-b-z)(1-at)}{zt}\phi(z,q,-tz^2)
   - (bz-b-z)\phi(z,q,-qtz^2)\;.
\end{multline}
\end{prop}
\begin{proof}
The recurrence (\ref{recurrence}) implies that $W(z;a,b;q,t)$ satisfies a 
functional equation,
$$
W(z;a,b;q,t) = z(1-z)(1-a b t)+z^2(1- a t - b q t) + z W(z;a,1;q,t) - z^2 b q t W(qz;a,1;q,t)\;.
$$
Letting $b=1$ and isolating $W(z;a,1,q,t)$ gives
$$
W(z;a,1;q,t)= \frac{z}{1-z}(1-at-zqt)-\frac{z^2}{1-z}qtW(qz;a,1;q,t)\;.
$$
This functional equation has the structure $W(z)=A(z)+B(z)W(qz)$ which can readily be solved by iteration to give
$W(z)=\sum_{n=0}^\infty A(q^nz)\prod_{k=0}^{n-1}B(q^kz)$. In this way we find
\begin{align*}
W(z;a,1;q,t)=&\sum_{n=0}^{\infty} \frac{(-1)^n z^{2n+1}q^{n^2+n}t^n
\left(1-at-zq^{n+1}t\right)}{(z;q)_{n+1}}\\
=&\frac{1-at}{t z} - 1 
 - \frac{1-at}{t z}\phi(z,q,-t z^2)
+  \phi(z,q,-qt z^2)\;.
\end{align*}
Inserting this expression into the functional equation gives (\ref{genfunct}).
\end{proof}

\begin{prop}
For $h\geq1$,
\begin{multline}
\label{orthopol}
Q_h(a,b;q,t)=
\sum\limits_{m=0}^\infty
(-t)^m q^{m(m-1)}\times\\
\left( (1-b) \qbinom{h-m}{m}{q} 
    +b \qbinom{h+1-m}{m}{q} 
     - (1-a)(1-b) \qbinom{h-1-m}{m-1}{q}
    - (1-a) b \qbinom{h-m}{m-1}{q} \right)\;.
\end{multline}
\end{prop}

\begin{proof}
We obtain $Q_h(a,b;q,t)$ by extracting the coefficient of $z^h$ in $W(z;a,b;q,t)$. 
We expand the $q$-product in the function 
$\phi$ with the help of the $q$-binomial theorem (see page 490 of \cite{andrews1999a-a}) to obtain
$$
\phi(z,q,tz^2) = 1 + \sum_{m=0}^{\infty} z^m \sum_{n=1}^\infty
q^{n(n-1)}
\qbinom{m-n-1}{n-1}{q} t^n\;.
$$
Inserting this expansion into (\ref{genfunct}) and collecting terms with equal powers
in $z$ gives (\ref{orthopol}).  
\end{proof}

Theorem 1 now follows from Propositions 4 and 6 and by checking that the expression gives the correct result also for $h=0$.

\section*{Acknowledgements} 
Financial support from the Australian Research Council via its support
for the Centre of Excellence for Mathematics and Statistics of Complex
Systems is gratefully acknowledged by the authors. A L Owczarek thanks
the School of Mathematical Sciences, Queen Mary, University of London
for hospitality.


\begin{thebibliography}{10}

\bibitem{andrews1999a-a}
G.~E. Andrews, R.~Askey, and R.~Roy,
\newblock {\em Special Functions},
\newblock volume~71 of {\em Encyclopedia of Mathematics and its Applications},
\newblock Cambridge University Press, Cambridge, 1999.

\bibitem{brak2005}
R.~Brak, A.~L. Owczarek, A.~Rechnitzer, and S.~Whittington,
\newblock {\em A directed walk model of a long chain polymer in a slit with attractive walls},
J. Phys. A {\bf 38}, 4309 (2005).


\bibitem{duchon2000}
P.~Duchon,
\newblock{\em On the Enumeration and Generation of Generalized Dyck Words},
Discrete Mathematics {\bf 225} (2000) 121-135.

\bibitem{ferrari2005}
L.~Ferrari and R.~Pinzani,
\newblock{\em Lattices of lattices paths},
J.~Statist.~Plann.~Inference {\bf 135} (2005) 7792.

\bibitem{flajolet1980} 
P.~Flajolet, 
\newblock {\em Combinatorial aspects of continued fractions},
Discrete Mathematics {\bf 41} (1980) 125-161.

\bibitem{flajolet2009} 
P.~Flajolet and R.~Sedgewick, 
\newblock {\em Analytic Combinatorics}, 
\newblock Cambridge University Press, Cambridge, 2009.

\bibitem{gasper2004} 
G.~Gasper and M.~Rahman, 
\newblock {\em Basic Hypergeometric Series},
\newblock volume~96 of {\em Encyclopedia of Mathematics and its Applications},
\newblock Cambridge University Press, Cambridge, 2004.

\bibitem{goulden1983} 
I.~P.~Goulden and D.~M.~Jackson, 
\newblock {\em Combinatorial Enumeration},
\newblock Wiley, New York, 1983.

\bibitem{sapounakis2006}
A.~Sapounakis, I.~Tasoulas, and P.~Tsikouras,
\newblock{\em On the Dominance Partial Ordering of Dyck Paths},
Journal of Integer Sequences {\bf 9} (2006) 06.2.5.

\bibitem{stanley1997} 
R.~P.~Stanley, 
\newblock {\em Enumerative Combinatorics, Volume 1},
\newblock volume~49 of {\em Cambridge Studies in Advanced Mathematics},
\newblock Cambridge University Press, Cambridge, 1999.

\bibitem{stanley1999} 
R.~P.~Stanley, 
\newblock {\em Enumerative Combinatorics, Volume 2},
\newblock volume~62 of {\em Cambridge Studies in Advanced Mathematics},
\newblock Cambridge University Press, Cambridge, 1999.

\end{thebibliography}
 \end{document}